\documentclass{article}

\usepackage{amsmath} 
\usepackage{amssymb}
 \usepackage{amsthm} 
 \usepackage{enumerate} 
 \usepackage[mathscr]{eucal}
 \usepackage{tikz} 
\usetikzlibrary{arrows}
\usepackage{graphicx}
\usepackage{caption}
\usepackage{subcaption}

\newcounter{cont}
\newtheorem{proposition}{Proposition}[section]
\newtheorem{definition}{Definition}[section]
\newtheorem{lemma}{Lemma}[section]
\newtheorem{theorem}{Theorem}[section]
\newtheorem{corollary}{Corollary}[section]
\newtheorem{remark}{Remark}[section]

\newcommand{\sem}[1]{\mathbf{#1}} 
\newcommand{\semd}[1]{\mathbf{{#1}_d}}
\newcommand{\dm}[1]{\mathbf{\hat{#1}}} 
\newcommand{\ol}[1]{\overline{#1}}

\title{The Riesz  hull of a semisimple MV-algebra} 
\author{Denisa Diaconescu and Ioana Leu\c{s}tean\\
{\small Department of Computer Science,} \\
{\small Faculty of Mathematics and Computer Science, University of Bucharest,}\\
{\small Academiei nr.14, sector 1, C.P. 010014,  Bucharest, Romania}\\   
{\small Emails: ddiaconescu@fmi.unibuc.ro,  ioana@fmi.unibuc.ro}}

\date{}

\begin{document}
\maketitle
\begin{center}
{\em  Dedicated to Prof. Antonio Di Nola on the occasion of his  65th birthday.}
\end{center}

\begin{abstract}
MV-algebras and Riesz MV-algebras are categorically equivalent to abelian lattice-ordered groups with strong unit and, respectively, with Riesz spaces (vector-lattices) with strong unit. A standard construction in the literature of lattice-ordered  groups is the {\em vector-lattice hull} of an archimedean lattice-ordered group. Following a similar approach, in this paper we define the 
{\em Riesz hull} of a semisimple MV-algebra.

\end{abstract}

\section{Introduction}

MV-algebras were first defined by Chang \cite{chang-58} as algebraic structures corresponding to the $\infty$-valued \L ukasiewicz logic. 
An {\em MV-algebra}  is a structure $(A,\oplus,^{*},0)$, where 
$(A,\oplus, 0)$ is an abelian monoid and the following identities hold 
for all $x,y\in A$: $(x^{*})^{*}=x$, $0^{*}\oplus x=0^{*}$ and 
 $(x^{*}\oplus y)^{*}\oplus y=(y^{*}\oplus x)^{*}\oplus x$. One of the main engines of  MV-algebra theory is the categorical equivalence between  MV-algebras and abelian lattice-ordered groups with strong unit \cite{Mundici-1986}.
As a consequence, any MV-algebra is isomorphic to the unit interval $[0,u]$ of an abelian lattice-ordered group $(G,u)$, with operations defined by
$x^*=u-x$ and  $x\oplus y=u\wedge(x+y)$.
MV-algebras stand to \L ukasiewicz logic as boolean algebras stand to classical logic: 
an equation holds in any MV-algebra if and only if it holds in the real interval  $[0,1]$ endowed with the following operations
\begin{center}
 $x\oplus y=\min\{1,x+y\}$ and $x^*=1-x$,
\end{center}
 for every $x,y\in [0,1]$. The real interval $[0,1]$ with the above operations is the {\em standard MV-algebra} and it is usually denoted by $[0,1]_{MV}$.

Adding a  product operation to the signature of MV-algebras was   a natural step, which  led to fruitful results, both in logic and algebra.  
Once  the MV-algebra structure is enriched, categorical  equivalences with particular lattice-ordered structures are  proved.

{\em  PMV-algebras} are defined  in \cite{dinola-dvurecenskij} as MV-algebras endowed with a product operation  $\cdot:A\times A\to A$, satisfying some particular identities. The category of PMV-algebras is equivalent with the category of   lattice-ordered rings with strong unit. In \cite{dinola-leustean-2013} the internal product  is replaced by a scalar multiplication with scalars from $[0,1]$, so  MV-algebras are endowed with a map $\cdot :[0,1]\times A\to A$. The structures  obtained in this way are called {\em Riesz MV-algebras} and they are categorically equivalent with Riesz spaces (vector-lattices) with strong unit. The real interval $[0,1]$ endowed  with the natural product  generates the variety of Riesz MV-algebras. Note that,  in the case of PMV-algebras,  $[0,1]$ generates only a proper quasi-variety \cite{Montagna}.

A standard construction in the literature of lattice-ordered  groups is the {\em vector-lattice hull} of an archimedean lattice-ordered group, 
defined by Conrad in \cite{Conrad-1971} and further analyzed by Bleier in \cite{Bleier-1971}. We also refer to \cite{JM-2002} for an extensive treatment of hull classes for archimedean lattice-ordered groups. 

We briefly  remind Conrad's definition.  If $G$ is an archimedean lattice-ordered group, then the {\em v-hull} of $G$ is a vector-lattice $U$ such  that $G$ is an essential subgroup of $U$ and no proper $\ell$-subspace of $U$ contains $G$. Assume 
$G_d$ is the divisible hull of $G$  and $\hat{G}_d$ is the Dedekind-MacNeille completion of $G_d$. Hence the vector-lattice generated by 
$G_d$ in $\hat{G}_d$, denoted by ${\mathbf R}(G)$, is the {\em v}-hull of $G$. Moreover, Bleier proved that the correspondence $G\mapsto {\mathbf R}(G)$ is functorial.

 In this paper  we investigate a similar construction for semisimple MV-algebras and semisimple  Riesz MV-algebras. If $A$ is a semisimple MV-algebra we say that a Riesz MV-algebra $U$ is the {\em Riesz hull} of $A$  if $A$ is essentially embedded in $U$ and $A$ is a set of generators for  $U$. In Section \ref{shull} we prove that   
{\em  any semisimple MV-algebra has a Riesz hull}. Moreover, the Riesz hull of the free MV-algebra over a set $X$ is the free Riesz
MV-algebra over $X$. In Section \ref{chull} we prove that the construction of the Riesz hull is functorial. Moreover, the  hull functor commutes with the categorical equivalences between the corresponding  classes of MV-algebras and lattice-ordered groups.

We chose to make direct proofs in the theory of MV-algebras.  
 Alternative proofs can be given   using Conrad's construction and various  preservation properties of the  categorical equivalence between MV-algebras and lattice-ordered groups, but we find the direct approach more relevant  for our purpose. 

In Section \ref{unu} and \ref{doi} we  recall the basic results on MV-algebras and Riesz MV-algebras that are required for our development. We refer to \cite{Darnel}  for background knowledge on lattice-ordered groups, to \cite{Zaanen-1971} for Riesz spaces and to 
\cite{Cohn-1965} for universal algebra.   

\section{MV-algebras}\label{unu}

\begin{definition}\begin{rm}
An {\em MV-algebra} is a structure $(A,\oplus,^*,0)$ of type (2,1,0)  which satisfies the following:
\begin{list}
{(MV\arabic{cont})}{\usecounter{cont}\setlength{\leftmargin}{0.5cm}}
	\item $(A,\oplus,0)$ is an abelian monoid,
	\item $({a^*})^* = a$,
	\item $0^* \oplus a = 0^*$,
	\item $(a^* \oplus b)^* \oplus b = (b^* \oplus a)^* \oplus a$,
\end{list}
for any $a,b \in A$.

We refer to \cite{cignoli-dottaviano-mundici} for all the unexplained notions related to MV-algebras.
\end{rm}\end{definition}

 In any MV-algebra $A$ we can define the following:
\begin{center}
\begin{tabular}{cc}
$1 \stackrel{def}{=} 0^*$, & $a \odot b \stackrel{def}{=} (a^* \oplus b^*)^*$,   \\

$a \vee b \stackrel{def}{=} (a \odot b^*) \oplus b$, & $a \wedge b \stackrel{def}{=} (a \oplus b^*) \odot b$, \\
\end{tabular}
\end{center}
for any $a,b \in A$. Hence $(A,\vee,\wedge, 0,1)$ is a bounded distributive lattice  such that
\begin{center}
$a\leq b$ if and only if $a\odot b^*=0$.
\end{center} 

The notions of MV-homomorphism and MV-subalgebra are defined as usual.

We recall that  a {\em lattice-ordered group} (an {\em  $\ell$-group})  is a structure $(G,+,0,\leq)$ such that $(G,+,0)$ is a group, $(G,\leq)$ is a lattice and  any group translation is isotone \cite{Darnel}. 
An element $u\in G$ is  a {\em strong unit} if $u\geq 0$ and for any $x\in G$ there is a natural number $n$ such that $x\leq nu$. An {\em $\ell u$-group}  will be an abelian $\ell$-group which has a strong unit. 
If $(G,u)$ is an $\ell u$-group,  we define
\begin{center}
  $[0,u]=\{x\in G\mid 0\leq x\leq u\}$ and

$x\oplus y=(x+y)\wedge u$,  $ x^*=u-x$,   for any $x, y\in [0,u]$.
\end{center}
Then $[0,u]_G=([0,u],\oplus,^*,0)$ is an MV-algebra.

We denote by $\mathcal{MV}$  the category of MV-algebras and by $\mathcal{AG}_u$  the category of unital abelian lattice-ordered groups with unit-preserving $\ell$-morphisms.
  In \cite{Mundici-1986} the functor $\Gamma\colon \mathcal{AG}_u\to \mathcal{MV}$   is defined as follows:
\begin{center}  
$\Gamma(G,u)=[0,u]_G$, for any unital $\ell$-group $(G,u)$, 

$\Gamma(f)=f|_{[0,u]}$, for any $\ell$-morphism $f:(G,u)\to (G^{\prime},u^\prime)$ from $\mathcal{AG}_u$.
\end{center}

\begin{theorem}\label{t:mundici}\cite{Mundici-1986}
 The functor $\Gamma$ establishes a categorical equivalence between  $\mathcal{AG}_u$ and $\mathcal{MV}$. 
\end{theorem}

The standard MV-algebra is $[0,1]=\Gamma({\mathbb R},1)$. 

\begin{theorem}\cite{chang-58}
An equation holds in $[0,1]$ if and only if it holds in any MV-algebra.
\end{theorem}
As a consequence, the variety of MV-algebras is generated by $[0,1]$.

\begin{theorem}\cite{DiNola-1991}\label{th-dinola}
Any MV-algebra $A$  is isomorphic with an algebra of $^*[0,1]$-valued functions, where $^*[0,1]$ is the unit interval of the 
lattice-ordered group of nonstandard reals $^*{\mathbb R}$.
\end{theorem}

If $A$ is an MV-algebra, $a\in A$ and $n\geq 0$ is a natural number, we define 
  \begin{center}
$0a\stackrel{def}{=}0$ and $na \stackrel{def}{=}(n-1)a\oplus a$,  if  $n>0$. 
\end{center}

\begin{definition}\begin{rm}
If $\iota:A\to B$ is an MV-embedding then we say that:
\begin{list}
{(\arabic{cont})}{\usecounter{cont}\setlength{\leftmargin}{0.5cm}}
	\item $\iota$ is {\em order dense} if for any $b>0$  in $B$, there exists $a>0$ in $A$ such that $\iota(a)\leq b$,
	\item $\iota$ is {\em essential} if for any $b>0$  in $B$, there exists $a>0$ in $A$ such that $\iota(a)\leq nb$, for some natural number  $n\geq 0$.
\end{list}
\end{rm}\end{definition}

For any MV-algebra $A$, a nonempty set $I \subseteq A$ is an {\em MV-ideal} if the following hold:
\begin{list}
{(I\arabic{cont})}{\usecounter{cont}\setlength{\leftmargin}{1cm}}
	\item $a \leq b$ and $b\in I$ implies $a\in I$,
	\item $a,b\in I$ implies $a\oplus b\in I$.
\end{list}

\begin{remark}\begin{rm}
 An embedding $\iota:A\to B$ is essential if and only if for any ideal $I$ of $B$, $I\neq \{0\}$ implies 
$I\cap \iota(A)\neq \{0\}$.
\end{rm}\end{remark}

\begin{lemma}\label{es}
Let $\iota:A\to B$  be an essential embedding. If $C$ is an MV-algebra and $f_A:A\to C$, $f_B:B\to C$ are MV-homomorphisms such that $f_B\circ\iota=f_A$  and  $f_A$ is an embedding then  $f_B$ is an embedding. 
\end{lemma}
\begin{proof} Assume $b\in B$ such that $f_B(b)=0$. If $b\neq 0$ there is $a>0$ in $A$ such that $a\leq nb$, so $f_A(a)=f_B(\iota (a))=0$. 
Since $f_A$ is an embedding we infer that $a=0$, which is a contradiction, so $b=0$ and $f_B$ is an embedding. 
\end{proof}

An ideal $I $ of $A$ is {\em proper} if $I\neq A$. A {\em maximal ideal} is a  maximal element of the set of proper ideals ordered by inclusion.  We denote by $Max(A)$ the set of all maximal ideals of $A$. Remember that, for any MV-algebra $A$,  $Max(A)$ endowed with the spectral topology is a compact and Hausdorff space \cite{cignoli-dottaviano-mundici}.

An MV-algebra is {\em semisimple} if  $\bigcap\{I\mid I\in Max(A)\}=\{0\}$. 

Recall that an $\ell u$-group $(G,u)$ is {\em archimedean} if, for any $x$, $y\in G$, we have
\begin{center}
 $nx\leq y$, for any $n\in {\mathbb N}$, implies $x\leq 0$.
\end{center}

\begin{remark}\begin{rm}\cite{cignoli-dottaviano-mundici}
Let $A$ be  an MV-algebra and $(G,u)$ an $\ell u$-group such that $A\simeq\Gamma(G,u)$. Then $A$ is semisimple if and only if $G$ is archimedean. 
\end{rm}\end{remark}

The semisimple MV-algebras are the algebras of $[0,1]$-valued functions, i.e. for any semisimple MV-algebra $A$ there exists  a set $X$ such that $A$ is isomorphic with a subalgebra of  $[0,1]^X$ \cite{Belluce-1986}. If $X$ is a topological space,  we set 
$C(X)=\{f:X\to [0,1]\mid f \mbox{ continuous}\}$,
which obviously is a semisimple MV-algebra. 

\begin{theorem}\cite{cignoli-dottaviano-mundici} Any semisimple MV-algebra $A$ is isomorphic with a separating subalgebra of  $C(Max(A))$. 
\end{theorem}

For a semisimple MV-algebra $A$ we denote by $\sem{A}$ the subalgebra of $C(Max(A))$  such that $A\simeq \sem{A}$ and by $\varphi_A:A\to \sem{A}$ the corresponding isomorphism.

\begin{definition}\begin{rm}
An MV-algebra $A$ is {\em divisible} if for any element  $a\in A$ and $n>1$ in ${\mathbb N}$ there exists $x\in A$ such that $nx=a$ and $(n-1)x\leq x^*$.\end{rm}\end{definition}
 
We refer to \cite{Gerla-2001} for a systematic investigation of the divisible MV-algebras and their logic. 

\begin{remark}\cite{Gerla-2001}\begin{rm}
An $\ell$-group $G$ is {\em divisible} if for any element $g\in G$ and any $n>1$ in ${\mathbb N}$ there exists $x\in G$ such that $nx=g$. One can easily see that  an $\ell u$-group 
$(G,u)$ is divisible if and only if the MV-algebra  $[0,u]_G$ is divisible.  
\end{rm}\end{remark}

If $X$ is a compact Hausdorff space then $C(X,{\mathbb R})=\{f:X\to {\mathbb R}\mid f \mbox{ continuous}\}$   is an $\ell$-group and the constant function ${\mathbf 1}$ is a  strong unit.

\begin{remark}\begin{rm}\label{doidiv}
It is well-known that  any MV-algebra can be embedded in a divisible one (see, for example, \cite{DiL-chapter}). 
We provide the details of this embedding for the semisimple case, which is relevant for our paper.

 Assume $(G,{\mathbf 1})$ is an $\ell u$-subgroup of $(C(X, {\mathbb R}), {\mathbf 1})$ and  $A=[0,{\mathbf 1}]_G\subseteq C(X)$.  We define
 $G_d=\{\frac{g}{n}\mid g\in G, n\in {\mathbb N}, n\neq 0\}$ and $A_d=[0,{\mathbf 1}]_{G_d}$. Hence $G_d$ is a divisible $\ell$-group and 
$A_d$ is a divisible MV-algebra.  
Let $g\in G$ and  $n\in {\mathbb N}$ such that $\frac{g}{n}\in A_d$. It follows that $0\leq g \leq  n{\mathbf 1}$ in $G$ , so there are 
$a_1$, $\cdots$, $a_{n}\in A$ such that $g=a_1+\cdots +a_{n}$. Hence $\frac{g}{n}=\frac{a_1}{n}+\cdots +\frac{a_{n}}{n}$. In consequence
\begin{center}
$A_d=\{a\in C(X)\mid a=\frac{a_1}{n}+\cdots +\frac{a_{n}}{n} \mbox{ for some } n\in {\mathbb N}, n\neq 0  \mbox{ and } a_1,\ldots, a_{n}\in A\}$,
\end{center}
and it is straightforward that $A\subseteq A_d$.
\end{rm}\end{remark}

If $X$ is a compact Hausdorff space and  $\sem{A}\leq C(X)$ is a semisimple MV-algebra then we get an embedding  
$\iota_{A,d}:\sem{A}\to\semd{A}$.

We note that

\begin{lemma}\label{divh}
Under the above hypothesis, the following properties hold.
\begin{list}
{(\alph{cont})}{\usecounter{cont}\setlength{\leftmargin}{0.5cm}}
\item The embedding ${\iota_{A,d}}$ is essential.
\item If $U$ is a semisimple divisible MV-algebra and $f:\sem{A}\to U$ is an MV-homomorphism then there exists a unique MV-homomorphism  $f_d:\semd{A}\to U$ such that $f_d(a)\circ\iota_{A,d}=f$. Moreover, if $f$ is an embedding then $f_d$ is also an embedding. 
\end{list}
\end{lemma}
\begin{proof}
(a) follows easily from the description of $\semd{A}$ from  Remark \ref{doidiv}.\\
(b) Assume that $G$ and $G_d$ are the $\ell u$-groups from Remark \ref{doidiv}. One can easily see that 
whenever $(H,v)$ is a divisible $\ell u$-group and $h:G\to H$ is an $\ell u$-morphism there exists a unique $\ell u$-morphism
$h^\#: G_d\to H$ extending $h$, which is simply defined by $h^\#(\frac{g}{n})=\frac{h(g)}{n}$ for any $g\in G$ and $n\in {\mathbb N}$. 
Hence the extension result for MV-algebras follows using the  functor $\Gamma$. The embeddings are preserved by (a) and Lemma \ref{es}.
\end{proof}

Recall that an MV-algebra $A$ is {\em complete} if any subset $\{a_i\mid i\in I\}$  of $A$ has infimum and supremum. 

\begin{definition}\begin{rm}
For a semisimple MV-algebra $A$ we say that $\hat{A}$ is  {\em  the Dedekind-MacNeille completion} of $A$ if  $A\leq \hat{A}$, $\hat{A}$ is complete and for any  element $\hat{a}\in \hat{A}$ there exists  a family $\{a_i\mid i\in I\}\subseteq A$ such that 
$\hat{a}=\bigvee\{a_i\mid i\in I\}$.   
\end{rm}\end{definition}

Since any complete MV-algebra is semisimple \cite[Proposition 6.6.2]{cignoli-dottaviano-mundici}, only semisimple MV-algebras admit completions.   

\begin{remark}\begin{rm}
Any semisimple MV-algebra $A$ has a Dedekind-MacNeille completion $\hat{A}$, which is unique up to isomorphism.  Moreover, 
$A$ is order dense in $\hat{A}$. 
\end{rm}\end{remark}

We refer to \cite{Ball-GG-Leustean} for a  study of  completions in the theory of  MV-algebras, with a special focus on the Dedekind-MacNeille completion.  

\section{Riesz MV-algebras}\label{doi}

Any MV-algebra is isomorphic  with the unit interval of an $\ell u$-group. If we consider a Riesz space  with strong unit instead of an $\ell u$-group, then the unit interval is closed under the scalar multiplication with scalars from $[0,1]$.  The structures obtained in this way are studied in \cite{dinola-leustean-2013}.

\begin{definition}\begin{rm}\cite{dinola-leustean-2013}
A {\em Riesz MV-algebra} is a structure $(V,\cdot,\oplus,^*,0)$, where $(V,\oplus,^*,0)$ is an MV-algebra and $\cdot : [0,1] \times V \rightarrow V$ is a function such that:
\begin{list}
{(RMV\arabic{cont})}{\usecounter{cont}\setlength{\leftmargin}{0.5cm}}
	\item $r \cdot (a\odot b^*)=(r\cdot a)\odot (r\cdot b)^*$,
	\item $(\max (r-q,0))\cdot a=(r\cdot a)\odot (q\cdot a)^*$,
	\item $(r \cdot q) \cdot a = r \cdot (q \cdot a)$,
	\item $1 \cdot a = a$,
\end{list}
for any $r,q \in [0,1]$ and $a,b\in V$.
\end{rm}\end{definition}

In order to simplify the notation, we shall frequently write $ra$ instead of $r \cdot a$, for any $r \in [0,1]$ and $a\in V$.
For a Riesz MV-algebra $(V,\cdot,\oplus,^*,0)$ we denote by ${\mathbf U}(V)=(V,\oplus,^*,0)$ its {\em MV-algebra reduct}.

\begin{remark}\cite{dinola-leustean-2013}
{\normalfont  If $V$ is a Riesz MV-algebra and $I\subseteq {\mathbf U}(V)$ is an MV-algebra ideal, then $r\cdot a\in I$ for any $r\in [0,1]$ and $a\in I$. Hence a Riesz MV-algebra has the same theory of ideals as its MV-algebra reduct. In consequence, a Riesz MV-algebra is semisimple if and only if its MV-algebra reduct is semisimple. } 
\end{remark}

\begin{proposition}\cite{dinola-leustean-2013} If $V_1$ and $V_2$ are Riesz MV-algebra and $f : {\mathbf U}(V_1) \rightarrow {\mathbf U}(V_2)$ is an MV-homomorphism, then $f(ra) = rf(a)$, for any $r\in [0,1]$ and $a\in V_1$.
\end{proposition}

\begin{remark}\cite{dinola-leustean-2013}\label{morf}
{\normalfont
By the previous proposition, it follows that {\em Riesz MV-algebra homomorphisms} are just MV-homomorphisms between Riesz MV-algebras, so we shall only state that a function is an MV-homomorphism, even if the domain and the codomain are Riesz MV-algebras.
}
\end{remark}

We recall that a {\em Riesz space} ({\em vector-lattice}) \cite{Zaanen-1971} is a structure 
$(L,\cdot, +,0,\leq)$
such that $(V,+,0,\leq)$ is an abelian $\ell$-group, $(V,\cdot,+,0)$ is a real vector space and, in addition,
 $x\leq y$ implies $r \cdot x\leq r\cdot y$,
for any $x$, $y\in L$ and $r\in{\mathbb R}$, $r\geq 0$. A Riesz space is {\em unital} if the underlaying  $\ell$-group  is {unital}.
If $(L,u)$ is a Riesz space with strong unit, then we denote by $\Gamma_R(L,u)=([0,u],\cdot,\oplus,^*,0)$, where 
$\cdot$ is the scalar multiplication restricted to scalars from $[0,1]$.

\begin{remark}\cite{dinola-leustean-2013} 
{\normalfont
For any unital Riesz space $(L,u)$, the structure 
$\Gamma_R(L,u)$ is a Riesz MV-algebra.
}
\end{remark}

In this way we can define a functor $\Gamma_R : \mathcal{RS}_u\to \mathcal{RMV}$, where  
$\mathcal{RS}_u$  is the category of unital Riesz spaces and $\mathcal{RMV}$ is the category of Riesz MV-algebras.  The categorial equivalence from Theorem \ref{t:mundici} leads to the following one.  

\begin{theorem}\cite{dinola-leustean-2013}
 The functor $\Gamma_R$ establishes a categorical equivalence.  
\end{theorem}

The standard Riesz MV-algebra is  $([0,1],\cdot, \oplus, ^*, 0)$ where $([0,1], \oplus, ^*, 0)$ is the standard MV-algebra and $\cdot$ is the product of real numbers.

\begin{theorem}\cite{dinola-leustean-2013} 
The variety of Riesz MV-algebras is generated by $[0,1]$.
\end{theorem}

\begin{lemma}\label{dmh}
The Dedekind-MacNeille completion of a semisimple divisible  MV-algebra $D$  is a Riesz MV-algebra $\hat{D}$ in which $D$ is order dense. 
\end{lemma}
\begin{proof}
It is a straightforward consequence of the fact that the  functor $\Gamma$ preserves both divisibility \cite{Gerla-2001} and 
completeness \cite{DiNola-Sessa}. Hence there exists a divisible $\ell u$-group $(G,u)$ such that $D\simeq \Gamma(G,u)$ and 
$\hat{D}=\Gamma(\hat{G},u)$, where $\hat{G}$ is  the Dedekind-MacNeille completion of $G$. Now we use the fact that 
the Dedekind-MacNeille completion of  a divisible abelian $\ell$-group is a Riesz space \cite{Darnel}.  The result can be directly proved by setting  $ra=\bigvee\{qa\mid q\in [0,1]\cap{\mathbb Q}, q\leq r\}$ for any $r\in [0,1]$ and $a\in \hat{D}$.
\end{proof}

\begin{remark}
{\normalfont
By Theorem \ref{th-dinola},  for any MV-algebra $A$ there exists a set $X$ such that $A$ is embedded in the MV-algebra 
$(^*[0,1])^X$.  Since $^*[0,1]$ is obviously a Riesz MV-algebra, one can easily see that   $(^*[0,1])^X$ becomes a Riesz MV-algebra with the scalar multiplication defined componentwise. Hence any MV-algebra can be embedded in a Riesz MV-algebra. 
}
\end{remark}

In the following we prove that,  for a semisimple MV-algebra $A$, we can define a unique (up to isomorphism) Riesz MV-algebra in which $A$ is essentially embedded and we will further analyze the properties of this embedding.

\section{The Riesz MV-algebra hull}\label{shull}

In the sequel, we follow closely the similar construction for archimedean $\ell$-groups from  \cite{Conrad-1971} and \cite{Bleier-1971}, but our proofs are made directly in the context of MV-algebras.

 Due to Remark \ref{morf},  in the rest of this paper we will make no distinction between MV-homomorphisms and Riesz MV-algebra homomorphisms. If $A$ is an MV-algebra and $X$ is a subset of $A$, we shall denote by $\langle X\rangle_{MV}$ the MV-subalgebra generated by $X$ in $A$. Similarly, if $V$ is a Riesz MV-algebra and $X$ is a subset of $V$, we shall denote by $\langle X\rangle_{RMV}$ the Riesz MV-subalgebra generated by $X$ in $V$.

If $A$ is a semisimple MV-algebra, then its divisible hull $A_d$  is also semisimple. If $X=Max(A_d)$ is the compact Hausdorff space of the maximal ideals of $A_d$, then  $$A\simeq \sem{A}\subseteq \semd{A}\subseteq C(X).$$ Let $\semd{\dm{A}}$  be the Dedekind-MacNeille completion of $\semd{A}$. By Lemma \ref{dmh}, $\semd{\dm{A}}$ is a Riesz MV-algebra. We denote by ${\mathbf R}(A)$ the Riesz MV-algebra generated by $\sem{A}$ in  $\semd{\dm{A}}$. 

For a semisimple MV-algebra $A$, we assume the following:

\begin{list}
{}{\usecounter{cont}\setlength{\leftmargin}{0.5cm}}
	\item $\varphi_A:A\to\sem{A}$ is the canonical MV-algebra isomorphism, 
	\item $\iota_{A,d}:\sem{A}\to \semd{A}$ is the  embedding of $\sem{A}$ in its divisible hull $\semd{A}$, 
	\item $\hat{\iota}_{A,d}:\semd{A}\to \semd{\dm{A}}$ is the  embedding of $\semd{A}$ in its Dedekind-MacNeille completion. 
\end{list}
Hence we denote by $\iota_A:A\to {\mathbf R}(A)$ the co-restriction to ${\mathbf R}(A)$ of the homomorphism
 $\hat{\iota}_{A,d}\circ {\iota}_{A,d}\circ\varphi_A$.

\begin{theorem}\label{main}
If $A$ is a semisimple MV-algebra and $\mathbf{R}(A)$ is defined as above, then the following properties hold.
\begin{list}
{(\alph{cont})}{\usecounter{cont}\setlength{\leftmargin}{0.5cm}}
	\item There exists an embedding $\iota_A:A\to \mathbf{R}(A)$ and $\mathbf{R}(A)=\langle\iota_A(A)\rangle_{RMV}$.
	\item The embedding ${\iota_A}$ is essential.
       \item  If  $V$ is a semisimple  Riesz MV-algebra and $f:A\to V$ is an MV-embedding then there exists an MV-embedding 
$f_R:\mathbf{R}(A)\to V$ such that $f_R(\iota_A(a))=f(a)$, for any $a\in A$.

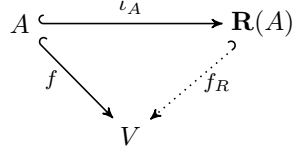
	\begin{center}
     	 \begin{tikzpicture}[
 		 font=\sffamily,
		  every matrix/.style={ampersand replacement=\&,column sep=1cm,row sep=1cm},
		  source/.style={thick,inner sep=.1cm},
		  to/.style={->,>=stealth',shorten >=.1pt,semithick,font=\sffamily\footnotesize}]
  
		 \matrix{
		    \node[source] (A) {$A$}; 
		    \&	
		    \& \node[source] (R) {$\mathbf{R}(A)$};   \\
    
		    \& \node[source] (tensorA) {$V$}; 
  		    \& \\
 		};
		
		\draw [right hook->,>=stealth',shorten >=.1pt,semithick,font=\sffamily\footnotesize]  (A) --node[midway,above] {$\iota_A$} (R) ; 
		\draw [right hook->,>=stealth',shorten >=.1pt,semithick,font=\sffamily\footnotesize]  (A) --node[midway,left] {$f$ $ $} (tensorA) ; 
		\draw [left hook->,>=stealth',shorten >=.1pt,semithick,font=\sffamily\footnotesize,dotted]  (R) --node[midway,right] {$f_R$} (tensorA) ; 
	\end{tikzpicture}
\end{center}
	\end{list}

\end{theorem}
\begin{proof}
${}$

\begin{list}
{}{\usecounter{cont}\setlength{\leftmargin}{0cm}}
\item (a) follows by definition, since $A$ is embedded in $\semd{A}$ and $\semd{A}$ is embedded in  $\semd{\dm{A}}$.  

\item (b) is a straightforward consequence of Lemmas \ref{divh} and \ref{dmh}. 

\item (c) Let $V$ be a semisimple Riesz MV-algebra and $f:A\to V$ an MV-embedding.  Since $V$ is also divisible, by Remark \ref{doidiv}, there is a unique MV-embedding $f_d: \semd{A}\to V$ such that 
$$f_d\circ\iota_{A,d}\circ\varphi_A=f.$$ If $\iota_V:V\to \dm{V}$ is the inclusion of $V$ in its Dedekind-MacNeille completion, then 
there exists a unique MV-embedding $\hat{f}_d:\semd{\dm{A}}\to \dm{V}$ such that 
$$\hat{f}_d\circ\hat{\iota}_{A,d}=\iota_V\circ f_d.
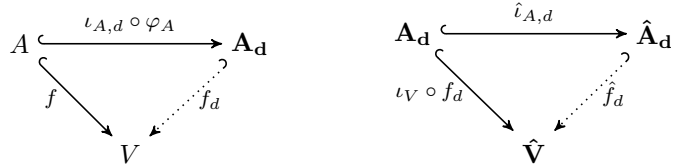
$$

\begin{center}
     	 \begin{tikzpicture}[
 		 font=\sffamily,
		  every matrix/.style={ampersand replacement=\&,column sep=1cm,row sep=1cm},
		  source/.style={thick,inner sep=.1cm},
		  to/.style={->,>=stealth',shorten >=.1pt,semithick,font=\sffamily\footnotesize}]
  
		 \matrix{
		    \node[source] (A) {$A$}; 
		    \&	
		    \& \node[source] (R) {$\semd{A}$};   \\
    
		    \& \node[source] (tensorA) {$V$}; 
  		    \& \\
 		};
		
		\draw [right hook->,>=stealth',shorten >=.1pt,semithick,font=\sffamily\footnotesize]  (A) --node[midway,above] {$\iota_{A,d}\circ\varphi_A$} (R) ; 
		\draw [right hook->,>=stealth',shorten >=.1pt,semithick,font=\sffamily\footnotesize]  (A) --node[midway,left] {$f$ $ $} (tensorA) ; 
		\draw [left hook->,>=stealth',shorten >=.1pt,semithick,font=\sffamily\footnotesize,dotted]  (R) --node[midway,right] {$f_d$} (tensorA) ; 

	\end{tikzpicture}
	\hspace{1cm}
     	 \begin{tikzpicture}[
 		 font=\sffamily,
		  every matrix/.style={ampersand replacement=\&,column sep=1cm,row sep=1cm},
		  source/.style={thick,inner sep=.1cm},
		  to/.style={->,>=stealth',shorten >=.1pt,semithick,font=\sffamily\footnotesize}]
  
		 \matrix{
		    \node[source] (A) {$\semd{A}$}; 
		    \&	
		    \& \node[source] (R) {$\semd{\dm{A}}$};   \\
    
		    \& \node[source] (tensorA) {$\dm{V}$}; 
  		    \& \\
 		};
		
		\draw [right hook->,>=stealth',shorten >=.1pt,semithick,font=\sffamily\footnotesize]  (A) --node[midway,above] {$\hat{\iota}_{A,d}$} (R) ; 
		\draw [right hook->,>=stealth',shorten >=.1pt,semithick,font=\sffamily\footnotesize]  (A) --node[midway,left] {$\iota_V\circ f_d$}(tensorA) ; 
		\draw [left hook->,>=stealth',shorten >=.1pt,semithick,font=\sffamily\footnotesize,dotted]  (R) --node[midway,right] {$\hat{f}_d$} (tensorA) ; 
		
	\end{tikzpicture}

\end{center}

It follows that 
\begin{center}
$\hat{f}_d\circ\iota_A=\hat{f}_d\circ\hat{\iota}_{A,d}\circ {\iota}_{A,d}\circ\varphi_A=
\iota_V\circ f_d\circ{\iota}_{A,d}\circ\varphi_A=\iota_V\circ f$,
\end{center}
and we get 
\begin{center}
 $\hat{f}_d({\mathbf R}(A))=\hat{f}_d(\langle \iota_A(A)\rangle_{RMV})=\langle \hat{f}_d(\iota_A(A))\rangle_{RMV}=$

$\langle \iota_V(f(A))\rangle_{RMV}=\langle f(A)\rangle_{RMV}\subseteq V$.
\end{center}

Therefore we define $f_R:{\mathbf R}(A)\to V$ as the co-restriction to $V$ of the restriction $\hat{f}_d|_{{\mathbf R}(A)}$. 
If $g:{\mathbf R}(A)\to V$ is another  MV-embedding such that $g\circ\iota_A=f$, then $g$ and $f$ coincide on the 
generators of ${\mathbf R}(A)$, so they coincide on ${\mathbf R}(A)$. 

\end{list}
\end{proof}

Following \cite{Conrad-1971}, we define the Riesz hull of an MV-algebra.

\begin{definition}\begin{rm}
 We say that a Riesz MV-algebra  $U$  is a {\em Riesz hull} of $A$ if there exists an essential  embedding $\eta:A\to U$ such that $U=\langle \eta(A)\rangle_{RMV}$.
\end{rm}\end{definition}

In consequence, Theorem  \ref{main} asserts that any semisimple MV-algebra has a Riesz hull which is unique, up to isomorphism.

\begin{corollary}\label{c1}
If $A$ is a semisimple MV-algebra, then ${\mathbf R}(A)\simeq{\mathbf R}(\semd{A})$.
\end{corollary}
\begin{proof}
It is a straightforward consequence of the construction.
\end{proof}

\begin{corollary}\label{c2}
If $A$ is a semisimple MV-algebra and $V$ is a semisimple Riesz MV-algebra such that 
$A\subseteq V$ and $V=\langle A\rangle_{RMV}$, then $V\simeq {\mathbf R}(A)$. 
 \end{corollary}
\begin{proof}
By Theorem \ref{main} (c), there exists an MV-embedding $e:{\mathbf R}(A)\to V$ such that $e(\iota_A(a))=a$ for any $a\in A$. Hence 
\begin{center} $e({\mathbf R}(A))=e(\langle \iota_A(A)\rangle_{RMV})=\langle e(\iota_A(A))\rangle_{RMV}=\langle A\rangle_{RMV}= V$, 
\end{center}
so $e$ is an isomorphism. 
\end{proof}

\begin{corollary}\label{c3}
Let $A$ be a semisimple MV-algebra and $V$ be a semisimple Riesz MV-algebra such that $A\subseteq V$ and $\langle A\rangle_{RMV}=V$. Then 
the embedding $A\hookrightarrow V$ is essential. If, in addition, $A$ is divisible, then the embedding $A\hookrightarrow V$ is order dense. 
\end{corollary}
\begin{proof}
The first part follows by Corollary \ref{c2} and Theorem \ref{main} (b). If $A$ is divisible, then $A\simeq \semd{A}$. In this case, the conclusion follows by the fact that $\semd{A}\subseteq{\mathbf R}(A)\subseteq\semd{\dm{A}}$ and Lemma \ref{dmh}.
\end{proof}

\begin{corollary}\label{c4}
If $V$ is a semisimple Riesz MV-algebra, then $V\simeq {\mathbf R}(V)$. In this case, $\iota_V$ is an isomorphism. 
\end{corollary}
\begin{proof} It follows from Corollary \ref{c2}.
\end{proof}

\begin{corollary}\label{c5}
Assume $V_1$ and $V_2$ are semisimple Riesz MV-algebras with the same MV-algebra reduct. Then $V_1\simeq V_2$.
\end{corollary}
\begin{proof} If $A$ is the MV-algebra reduct of $V_1$ and $V_2$ then, by Corollary \ref{c2},  we get 
$V_1\simeq {\mathbf R}(A)\simeq V_2$.
\end{proof}

The above result asserts that, given an MV-algebra $A$, there is at most one structure, up to isomorphism, of Riesz MV-algebra with the MV-algebra reduct $A$. 

In the sequel we prove that the Riesz MV-algebra hull preserves freeness.
For a nonempty  set $X$, we shall denote by $Free_{MV}(X)$ the free MV-algebra over $X$ and by $Free_{RMV}(X)$ the free Riesz MV-algebra over  $X$. The free algebras exist in the classes of MV-algebras and Riesz MV-algebras since both classes are varieties. 

\begin{proposition}\label{free-hull}
For any nonempty set $X$, ${\mathbf R}(Free_{MV}(X))\simeq Free_{RMV}(X)$. Therefore, the free MV-algebra generated by $X$ is  essentially embedded in the free Riesz MV-algebra generated by $X$. Moreover, the embedding can be chosen to be an inclusion.
\end{proposition} 
\begin{proof}
If  $T=[0,1]^{[0,1]^X}$ then $T$ is a Riesz MV-algebra with the operations defined component-wise. For any $x\in X$ we denote by $\pi_x\in T$ the corresponding projection function and we set $\tilde{X}=\{\pi_x\mid x\in X\}$.
Since the variety of MV-algebras is generated by $[0,1]_{MV}$ and the variety of Riesz MV-algebras is generated by $[0,1]_{RMV}$, 
by general properties in universal algebra, $Free_{MV}(X)$ is the MV-algebra generated by $\tilde{X}$ in $T$ and      
$Free_{RMV}(X)$ is the Riesz  MV-algebra generated by $\tilde{X}$ in $T$. We have that $Free_{MV}(X)= \langle \tilde{X}\rangle_{MV}$ and
\begin{center}
  $Free_{RMV}(X)= \langle \tilde{X}\rangle_{RMV}=\langle Free_{MV}(X)\rangle_{RMV}$.
\end{center}
The conclusion follows from Corollary \ref{c2}.
\end{proof}

\section{Categorical setting: the functor $\mathbf R$}\label{chull}

The main step for  obtaining  a functorial setting is to prove a general extension result for  morphisms, as which we do in Proposition \ref{ext-morf}. The results of this section follow closely the ideas from \cite{Bleier-1971}. 

\begin{remark}\begin{rm}\label{freeA}
Let $A$ be a semisimple MV-algebra and $X\subset A$ such that $\langle X\rangle_{MV}=A$. Using Proposition \ref{free-hull},  the free MV-algebra generated by $X$ is essentially included  in the free Riesz MV-algebra generated by $X$ and we denote this inclusion by
$\iota_X:Free_{MV}(X)\to Free_{RMV}(X)$.  
Let $\alpha:Free_{MV}(X)\to A$ be the unique MV-homomorphism such that $\alpha(x)=x$ for any $x\in X$ and 
$\ol{\alpha}:Free_{RMV}(X)\to R(A)$ be the unique MV-homomorphism such that $\ol{\alpha}(x)=\iota_A(x)$ for any $x\in X$.

\begin{center}
      \begin{tikzpicture}[
  font=\sffamily,
  every matrix/.style={ampersand replacement=\&,column sep=2cm,row sep=1cm},
  source/.style={thick,inner sep=.1cm},
  to/.style={->,>=stealth',shorten >=.1pt,semithick,font=\sffamily\footnotesize}]
  
 \matrix{
    \node[source] (A) {$Free_{MV}(X)$}; 
    \& \node[source] (tensorA) {$Free_{RMV}(X)$};   \\
    
    \node[source] (B) {$A$}; 
    \& \node[source] (tensorB) {${\mathbf R}(A)$};   \\
 };
 
\draw [to]  (A) -- node[midway,left] {$\alpha$} (B) ;   

\draw [right hook->,>=stealth',shorten >=.1pt,semithick,font=\sffamily\footnotesize]  (A) --node[midway,above] {$\iota_X$} (tensorA) ; 
\draw [right hook->,>=stealth',shorten >=.1pt,semithick,font=\sffamily\footnotesize]  (B) --node[midway,above] {$\iota_A$} (tensorB) ;   
\draw [->,>=stealth',shorten >=.1pt,semithick,font=\sffamily\footnotesize,dotted] (tensorA) -- node[midway,right] {$\overline{\alpha}$} (tensorB) ;
\end{tikzpicture}
  \end{center}
\end{rm}\end{remark}

\begin{proposition}\label{help1}
Under the above hypothesis, the following properties hold:
\begin{list}
{(\alph{cont})}{\usecounter{cont}\setlength{\leftmargin}{0.5cm}}
\item $\ol{\alpha}\circ\iota_X=\iota_A\circ\alpha$,
\item $\alpha$ and $\ol{\alpha}$ are surjective,
\item $\ker{\ol{\alpha}}=\bigcap\{J\mid J \in {\mathcal J}\}$, where
\begin{center}
${\mathcal J}=\left\{J\subseteq  Free_{RMV}(X)\mid J \mbox{  ideal},\,\,  \iota_X(\ker{\alpha})\subseteq J,\right.$

$\left.  \mbox{and }Free_{RMV}(X)/_J \mbox{ is semisimple}\right\}$.
\end{center}
 
\end{list}
\end{proposition}
\begin{proof}
${}$

\begin{list}
{}{\usecounter{cont}\setlength{\leftmargin}{0.5cm}}
\item (a) $(\ol{\alpha}\circ\iota_X)(x)=\iota_A(x)=(\iota_A\circ\alpha)(x)$ for any $x\in X$, so the morphisms coincide on generators.
\item (b) $\alpha(Free_{MV}(X))=\alpha(\langle  X\rangle  _{MV})=\langle  \alpha(X)\rangle  _{MV}=A$ and 
\begin{center}
$\ol{\alpha}(Free_{RMV}(X))=\ol{\alpha}(\langle  Free_{MV}(X)\rangle  _{RMV})=$ 

$\langle  \iota_A(\alpha(Free_{MV}(X)))\rangle  _{RMV}=\langle  \iota_A(A)\rangle  _{RMV}={\mathbf R}(A)$.
\end{center}
\item (c) If $z\in\ker{\alpha}$ then $\iota_A(\alpha(z))=0$, so $\ol{\alpha}(\iota_X(z))=0$. It follows that 
$\iota_X(\ker{\alpha})\subseteq \ker{\ol{\alpha}}$. In fact, we have $\iota_X(\ker{\alpha})=\ker{\ol{\alpha}}\cap \iota_X(Free_{MV}(X))$. We set
  $\ol{J}=\bigcap\{J\mid J\in {\mathcal J}\}$ and $F=Free_{RMV}(X)/_{\ol{J}}$. By a general result of universal algebra \cite[Proposition 7.1]{Cohn-1965}, $F$ is isomorphic with a subdirect product of the family $\{Free_{RMV}(X)/_J\mid J\in {\mathcal J}\}$, so $F$ is a subalgebra of a direct product of semisimple MV-algebras. Therefore, $F$ is a semisimple MV-algebra.  
If we set  $M=\{y/_{\ol{J}}\mid y\in \iota_X(Free_{MV}(X))\}$ then $\langle  M\rangle  _{RMV}=F$, so ${\mathbf R}(M)=F$ by Corollary \ref{c3} and 
the inclusion $M\subseteq F$ is essential. 

It is clear that $\iota_X(\ker{\alpha})\subseteq \ol{J}\subseteq\ker{\ol{\alpha}}$. 
In order to prove that $\ol{J}=\ker{\ol{\alpha}}$ , we assume that
there exists an element $z\in \ker{\ol{\alpha}}\setminus\ol{J}$. Hence $z/_{\ol{J}}\neq 0$ in  $F$. 
Since the inclusion $M\subseteq F$ is essential it follows that there exists an element $y\in \iota_X(Free_{MV}(X))$ such that 
$ 0< y/_{\ol{J}}\leq nz/_{\ol{J}}$. 
Note that $y/_{\ol{J}}\neq 0$ in $F$.
We denote $w=y\odot(nz)^*$, so $w\in \ol{J}$ and $y\leq (nz)\vee (y)=(nz)\oplus w$. Note that $w\in \ol{J}\subseteq\ker{\ol{\alpha}}$ and 
$z\in\ker{\ol{\alpha}}$, so we get $y\in \ker{\ol{\alpha}}$. But $y\in \iota_X(Free_{MV}(X))$, so $y\in \iota_X(Free_{MV}(X))\cap \ker{\ol{\alpha}}=\iota_X(\ker{\alpha})$. Since $\iota_X(\ker{\alpha})\subseteq\ol{J}$, it follows that $y/_{\ol{J}}=0$ in $F$, which is  a contradiction. 
\end{list}
\end{proof}

\begin{proposition}\label{ext-morf} 
Let $A$ be a semisimple MV-algebra.
For any semisimple Riesz MV-algebra $V$ and  for any  MV-homomorphism $f:A\to V$ there exists a unique MV-homomorphism 
$f_R:\mathbf{R}(A)\to V$ such that $f_R(\iota_A(a))=f(a)$, for any $a\in A$.
\begin{center}
     	 \begin{tikzpicture}[
 		 font=\sffamily,
		  every matrix/.style={ampersand replacement=\&,column sep=1cm,row sep=1cm},
		  source/.style={thick,inner sep=.1cm},
		  to/.style={->,>=stealth',shorten >=.1pt,semithick,font=\sffamily\footnotesize}]
  
		 \matrix{
		    \node[source] (A) {$A$}; 
		    \&	
		    \& \node[source] (R) {$\mathbf{R}(A)$};   \\
    
		    \& \node[source] (tensorA) {$V$}; 
  		    \& \\
 		};
		\draw [right hook->,>=stealth',shorten >=.1pt,semithick,font=\sffamily\footnotesize]  (A) --node[midway,above] {$\iota_A$} (R) ; 
		\draw [to]  (A) -- node[midway,left] {$f$} (tensorA) ;     
		\draw [<-,>=stealth',shorten >=.1pt,semithick,font=\sffamily\footnotesize,dotted] (tensorA) -- node[midway,right] {$f_R$} (R) ;
	\end{tikzpicture}
\end{center}
\end{proposition}
\begin{proof}

Assume $V$ is a semisimple Riesz MV-algebra and $f:A\to V$ is an MV-homomorphism. We consider  $X\subseteq A$ such that $\langle X\rangle_{MV}=A$ and we define $\alpha:Free_{MV}(X)\to A$ and $\ol{\alpha}:Free_{RMV}(X)\to {\mathbf R}(A)$ as in Remark \ref{freeA}.  Let $\ol{f}:Free_{RMV}(X)\to V$ be the unique MV-homomorphism such that $\ol{f}(x)=f(x)$ for any $x\in X$.
By Proposition \ref{help1} (b), we infer that $Free_{RMV}(X)/_{\ker{\ol{\alpha}}}\simeq {\mathbf R}(A)$ and we can safely identify them.

\begin{center}
  \begin{tikzpicture}[
  font=\sffamily,
  every matrix/.style={ampersand replacement=\&,column sep=2cm,row sep=1cm},
  source/.style={thick,inner sep=.25cm},
  to/.style={->,>=stealth',shorten >=.2pt,semithick,font=\sffamily\footnotesize}]
  
 \matrix{
   \& \node[source](FMV){$Free_{MV}(X)$}; \& \\
   \node[source](A){$A$}; \& \node[source](RA){$R(A)$}; \& \node[source](FRMV){$Free_{RMV}(X)$}; \\
   \& \node[source](V){$V$}; \& \\ 	
 };

\draw [to]  (FMV) -- node[midway,left] {$\alpha$} (A) ;    
\draw [right hook->,>=stealth',shorten >=.2pt,semithick,font=\sffamily\footnotesize]  (FMV) --node[midway,right] {$\iota_X$} (FRMV) ;
\draw [right hook->,>=stealth',shorten >=.2pt,semithick,font=\sffamily\footnotesize]  (A) --node[midway,above] {$\iota_A$} (RA) ;
\draw [to]  (FRMV) -- node[midway,above] {$\overline{\alpha}$} (RA) ;
\draw [to]  (A) -- node[midway,left] {$f$} (V) ;
\draw [to]  (FRMV) -- node[midway,right] {$\overline{f}$} (V) ;    
\draw [->,>=stealth',shorten >=.1pt,semithick,font=\sffamily\footnotesize,dotted] (RA) -- node[midway,right] {$f_R$} (V) ;
 
\end{tikzpicture}
\end{center}

%
%

We note that $\ol{f}(Free_{RMV}(X))$ is a Riesz  MV-subalgebra of $V$, so it is semisimple. Therefore, by Proposition \ref{help1} (c) it follows that $\ker{\ol{\alpha}}\subseteq \ker{\ol{f}}$, so there exists a unique  MV-homomorphism $f_R:{\mathbf R}(A)\to V$ such that $f_R\circ\ol{\alpha}=\ol{f}$. It follows that
\begin{center}
$f_R\circ{\iota_A}\circ\alpha=f_R\circ\ol{\alpha}\circ\iota_X=\ol{f}\circ\iota_X=f\circ\alpha$.
 \end{center}
Since $\alpha$ is surjective, we get $f_R\circ\iota_A=f$. 

In order to prove the uniqueness, assume that $g:{\mathbf R}(A)\to V$ is an MV-homomorphism such that $g\circ\iota_A=f$. It follows that 
$g\circ\ol{\alpha}\circ\iota_X=g\circ\iota_A\circ\alpha=f\circ\alpha=\ol{f}\circ\iota_X$ and we get 
$g\circ\ol{\alpha}=\ol{f}$, since they coincide on the generators of $Free_{RMV}(X)$.  We proved that $g$ satisfies the property that uniquely defines $f_R$, so $g=f_R$.
\end{proof}

\begin{lemma}\label{RMV-homomorphism}\label{RMV-homo}
Let $A$ and $B$ be semisimple MV-algebras.  For any homomorphism $h: A \rightarrow B$, there is a unique  homomorphism ${\mathbf R}(h): {\mathbf R}(A) \to {\mathbf R}(B)$ such that $${\mathbf R}(h)\circ \iota_A = \iota_B \circ h.$$ In addition, if $h$ is an embedding, then ${\mathbf R}(h)$ is also an embedding. 

  \begin{center}
      \begin{tikzpicture}[
  font=\sffamily,
  every matrix/.style={ampersand replacement=\&,column sep=2cm,row sep=1cm},
  source/.style={thick,inner sep=.1cm},
  to/.style={->,>=stealth',shorten >=.1pt,semithick,font=\sffamily\footnotesize}]
  
 \matrix{
    \node[source] (A) {$A$}; 
    \& \node[source] (tensorA) {${\mathbf R}(A)$};   \\
    
    \node[source] (B) {$B$}; 
    \& \node[source] (tensorB) {${\mathbf R}(B)$};   \\
 };
\draw [to]  (A) -- node[midway,left] {$h$} (B) ;   
\draw [right hook->,>=stealth',shorten >=.1pt,semithick,font=\sffamily\footnotesize]  (A) -- node[midway,above] {$\iota_A$}  (tensorA) ; 
\draw [right hook->,>=stealth',shorten >=.1pt,semithick,font=\sffamily\footnotesize]  (B) -- node[midway,above] {$\iota_B$}  (tensorB) ; 
\draw [->,>=stealth',shorten >=.1pt,semithick,font=\sffamily\footnotesize,dotted] (tensorA) -- node[midway,right] {${\mathbf R}(h)$} (tensorB) ;
\end{tikzpicture}
  \end{center}
\end{lemma}
\begin{proof}
We apply Proposition \ref{ext-morf}  for $V= {\mathbf R}(B)$ and for $f=\iota_B\circ h$. Therefore ${\mathbf R}(h)=f_R$.  
\end{proof}

We consider the forgetful functor between the category $\mathcal{RMV}_s$ of semisimple RMV-algebras and the category $\mathcal{MV}_s$ of semisimple  MV-algebras:
\begin{center}
${\mathbf U} : \mathcal{RMV}_s \rightarrow \mathcal{MV}_s$,
\end{center}
which forgets the scalar multiplication.

We also define the functor $${\mathbf R}: \mathcal{MV}_s \rightarrow \mathcal{RMV}_s$$  as follows:
\begin{list}
{$\cdot$}{\usecounter{cont}\setlength{\leftmargin}{0.5cm}}
	\item for any semisimple MV-algebras $A$, ${\mathbf R}(A) $ is the Riesz hull of $A$, 
	\item for any MV-homomorphism $h:A \rightarrow B$, ${\mathbf R}(h)$ is the unique homomorphism such that ${\mathbf U}({\mathbf R}(h)) \circ \iota_A = \iota_B \circ h$.
\end{list}

\begin{theorem}
Under the above settings, $({\mathbf R}, {\mathbf U})$ is an adjoint pair.
\end{theorem}
\begin{proof}
One can easily see that ${\mathbf R}$ is a functor. If $A$ is  a semisimple MV-algebra we define
$\eta_A:A\to{\mathbf U}({\mathbf R}(A))$, $\eta_A(a)=\iota_A(a)$ for any $a\in A$.
If $V$ is a semisimple Riesz MV-algebra, let $\varepsilon_V=\iota_V^{-1}:{\mathbf R}({\mathbf U}(V))\to V$. By Corollary \ref{c4}, $\varepsilon_V$ is  an isomorphism.   

In order to prove that $\mathbf R$ is a left adjoint to $\mathbf U$, we have to prove the following properties, for any MV-algebra $A$ and Riesz MV-algebra $V$:
\begin{list}
{(\arabic{cont})}{\usecounter{cont}\setlength{\leftmargin}{0.5cm}}
	\item for any $f\in \mathcal{MV}_s(A,{\mathbf U}(V))$,  there exists $g\in \mathcal{RMV}_s({\mathbf R}(A), V)$ such that 
$${\mathbf U}(g)\circ \eta_A=f,$$
	\item for any   $g\in \mathcal{RMV}_s({\mathbf R}(A), V)$   there exists $f\in \mathcal{MV}_s(A,{\mathbf U}(V))$ such that 
$$\varepsilon_V\circ{\mathbf R}(f)=g.$$
\end{list}
 
\begin{center}
\begin{tikzpicture}[
 		 font=\sffamily,
		  every matrix/.style={ampersand replacement=\&,column sep=1cm,row sep=1cm},
		  source/.style={thick,inner sep=.1cm},
		  to/.style={->,>=stealth',shorten >=.1pt,semithick,font=\sffamily\footnotesize}]
  
		 \matrix{
		    \node[source] (A) {$A$}; 
		    \&	
		    \& \node[source] (R) {${\mathbf U}({\mathbf R}(A))$};   \\
    
		    \& \node[source] (tensorA) {${\mathbf U}(V)$}; 
  		    \& \\
 		};
		\draw [to]  (A) -- node[midway,above] {$\eta_A$} (R) ;   
		\draw [to]  (A) -- node[midway,left] {$f$} (tensorA) ;     
		\draw [<-,>=stealth',shorten >=.1pt,semithick,font=\sffamily\footnotesize,dotted] (tensorA) -- node[midway,right] {${\mathbf U}(g)$} (R) ;
	\end{tikzpicture}
     	 \begin{tikzpicture}[
 		 font=\sffamily,
		  every matrix/.style={ampersand replacement=\&,column sep=1cm,row sep=1cm},
		  source/.style={thick,inner sep=.1cm},
		  to/.style={->,>=stealth',shorten >=.1pt,semithick,font=\sffamily\footnotesize}]
  
		 \matrix{
		    \node[source] (A) {${\mathbf R}({\mathbf U}(V))$}; 
		    \&	
		    \& \node[source] (R) {$V$};   \\
    
		    \& \node[source] (tensorA) {${\mathbf R}(A)$}; 
  		    \& \\
 		};
		\draw [to]  (A) -- node[midway,above] {$\varepsilon_V$} (R) ;   
		\draw [->,>=stealth',shorten >=.1pt,semithick,font=\sffamily\footnotesize,dotted]  (tensorA) -- node[midway,left] {${\mathbf R}(f)$} (A) ;     
		\draw [to] (tensorA) -- node[midway,right] {$g$} (R) ;
	\end{tikzpicture}
\end{center}

The property (1) follows by Proposition \ref{ext-morf} with $g=f_R$ whenever $f\in \mathcal{MV}_s(A,{\mathbf U}(V))$.  In order to prove (2),  assume that  $g\in \mathcal{RMV}_s({\mathbf R}(A), V)$ and set $f={\mathbf U}(g)\circ\iota_A$. Hence ${\mathbf R}(f)$ is the unique homomorphism such that $${\mathbf U}({\mathbf R}(f))\circ\iota_A=\iota_{{\mathbf U}(V)}\circ f.$$  Therefore we have ${\mathbf U}({\mathbf R}(f))\circ\iota_A=\iota_{{\mathbf U}(V)}\circ {\mathbf U}(g)\circ\iota_A. $ Since
$\iota_A$ is an embedding, we get that ${\mathbf U}({\mathbf R}(f))=\iota_{{\mathbf U}(V)}\circ {\mathbf U}(g) $.  

We note that $\iota_{{\mathbf U}(V)}={\mathbf U}(\iota_V)={\mathbf U}(\varepsilon_V^{-1})$. It follows that 
${\mathbf U}(\varepsilon\circ{\mathbf R}(f))={\mathbf U}(g)$, so $\varepsilon\circ{\mathbf R}(f)=g$. 
\end{proof}

In the following we prove that the hull functor ${\mathbf R}$ and the functor $\Gamma$ commute.

\begin{remark}\label{helpG}
{\normalfont
Let $A$ be a semisimple MV-algebra and $(G,u)$ an $\ell u$-group such that $A=\Gamma(G,u)$. If $X\subseteq A$  and $\langle X\cup\{u\}\rangle_{\ell}$ is the $\ell $-group generated by $ X\cup\{u\}$ in $G$, then we note that $(\langle X\cup\{u\}\rangle_{\ell},u)$ is an $\ell u$-subgroup of $(G,u)$. As a consequence,  we infer that $\Gamma(\langle X\cup\{u\}\rangle_{\ell},u)$ is an MV-subalgebra of $A$.  It is now straightforward that 
$\langle X\rangle_{MV}=\Gamma(\langle X\cup\{u\}\rangle_{\ell},u)$.  Assume now that  $V$ is a Riesz MV-algebra, $(H,u)$ is a unital Riesz space such that $\Gamma_R(H,u)=V$  and $X\subseteq V$. It is straightforward that 
$\langle X\rangle_{RMV}=\Gamma_R(\langle X\cup\{u\}\rangle_{v\ell},u)$, where  $\langle X\cup\{u\}\rangle_{v\ell}$ is the Riesz space generated by $X\cup\{u\}$ in $H$.

}
\end{remark}

\begin{proposition}\label{comG}
If $A$ is a semisimple MV-algebra and $(G,u)$ an $\ell u$-group such that $A=\Gamma(G,u)$, then 
${\mathbf R}(A)=\Gamma_R({\mathbf R}(G), u)$.
\end{proposition}
\begin{proof} 
We recall that $\semd{A}=\Gamma(\semd{G},u)$, so $\semd{A}\subseteq\semd{G}\subseteq \semd{\hat{G}}$. Moreover, $\semd{\hat{G}}$ is a Riesz space and ${\mathbf R}(G)= \langle \semd{G}\rangle_{v\ell}$, i.e. ${\mathbf R}(G)$ is the Riesz space generated by $\semd{G}$ in $\semd{\hat{G}}$.
Following  Remark \ref{helpG} we have ${\mathbf R}(A)=\langle \semd{A}\rangle_{RMV}=\Gamma_R(\langle \semd{A}\cup\{u\}\rangle_{v\ell},u)$. Since $u\in\semd{A}$ and $\langle \semd{A}\rangle_{v\ell}=\langle \semd{G}\rangle_{v\ell}$,  we get
\begin{center}
${\mathbf R}(A)=\Gamma_R(\langle \semd{G}\rangle_{v\ell},u)=\Gamma_R({\mathbf R}(G),u)$.
\end{center}
\end{proof}

We denote by $\mathcal{AG}_{ua}$ the category of archimedean $\ell u$-groups and by $\mathcal{RS}_{ua}$ the category of archimedean  Riesz spaces with strong unit. By \cite{Bleier-1971}, the correspondence $G\mapsto{\mathbf R}(G)$, which associates to an $\ell$-group its {\em v}-hull,   is functorial. If $G$ has a strong unit $u$, following Conrad's construction, one can easily see that $u$ is also a strong unit of ${\mathbf R}(G)$. Hence we get a functor 
 ${\mathbf R}:\mathcal{AG}_{ua}\to \mathcal{RS}_{ua}$.

\begin{theorem}
The following diagram is commutative:
\begin{center}
      \begin{tikzpicture}[
  font=\sffamily,
  every matrix/.style={ampersand replacement=\&,column sep=2cm,row sep=1cm},
  source/.style={thick,inner sep=.1cm},
  to/.style={->,>=stealth',shorten >=.1pt,semithick,font=\sffamily\footnotesize}]
  
 \matrix{
    \node[source] (A) {$\mathcal{AG}_{ua}$}; 
    \& \node[source] (tensorA) {$\mathcal{MV}_s$};   \\
    
    \node[source] (B) {$\mathcal{RS}_{ua}$}; 
    \& \node[source] (tensorB) {$\mathcal{RMV}_s$};   \\
 };
\draw [to]  (A) -- node[midway,left] {${\mathbf R}$} (B) ;   
\draw [to]  (A) -- node[midway,above] {$\Gamma$}  (tensorA) ; 
\draw [to]  (B) -- node[midway,above] {$\Gamma_R$}  (tensorB) ;  
\draw [to] (tensorA) -- node[midway,right] {${\mathbf R}$} (tensorB) ;
\end{tikzpicture}
  \end{center}
\end{theorem}
\begin{proof}
It is a straightforward consequence of Proposition \ref{comG}.
\end{proof}

\bigskip

{\bf Acknowledgment.} I. Leu\c stean  was partially  supported by the strategic grant \mbox{POSDRU/89/1.5/S/58852,}  cofinanced by ESF within SOP HRD 2007-2013.

\bibliography{IL-DD-bib-rev}

\begin{thebibliography}{10}

\bibitem{Ball-GG-Leustean}
R.~Ball, G.~Georgescu, and I.~Leu\c{s}tean.
\newblock Cauchy completions of {MV}-algebras.
\newblock {\em Algebra Universalis}, 47:367--407, 2002.

\bibitem{Belluce-1986}
L.P. Belluce.
\newblock Semisimple algebras of infinite valued logic and bold fuzzy set
  theory.
\newblock {\em Canadian Journal of Mathematics}, 38(6):1356--1379, 1986.

\bibitem{Bleier-1971}
R.D. Bleier.
\newblock Minimal vector lattice covers.
\newblock {\em Bull. Austral. Math. Soc.}, 5:331--335, 1971.

\bibitem{chang-58}
C.C. Chang.
\newblock Algebraic analysis of many valued logics.
\newblock {\em Trans. A.M.S.}, 88:467--490, 1958.

\bibitem{cignoli-dottaviano-mundici}
R.~Cignoli, I.M.L. D'Ottaviano, and D.~Mundici.
\newblock {\em Algebraic Foundations of Many-Valued Reasoning}.
\newblock Kluwer Academic, 2000.

\bibitem{Cohn-1965}
P.M. Cohn.
\newblock {\em Universal Algebra}.
\newblock Harper \& Row, 1965.

\bibitem{Conrad-1971}
P.F. Conrad.
\newblock Minimal vector lattice covers.
\newblock {\em Bull. Austral. Math. Soc.}, 4:35--39, 1971.

\bibitem{Darnel}
M.R. Darnel.
\newblock {\em Theory of Lattice-Ordered Groups}.
\newblock Marcel Dekker, Inc., 1995.

\bibitem{DiNola-1991}
A.~Di~Nola.
\newblock Representation and reticulation by quotients of {MV}-algebras.
\newblock {\em Ricerche di Matematica}, 40(2):291--297, 1991.

\bibitem{dinola-dvurecenskij}
A.~Di~Nola and A.~Dvurecenskij.
\newblock Product {MV}-algebras.
\newblock {\em Multiple-Valued Logics}, 6:193--215, 2001.

\bibitem{dinola-leustean-2013}
A.~Di~Nola and I.~Leu\c{s}tean.
\newblock {\L}ukasiewicz logic and {R}iesz spaces.
\newblock {\em Soft Computing, to appear, {\rm 2013}}.

\bibitem{DiL-chapter}
A.~Di~Nola and I.~Leu\c{s}tean.
\newblock {\em Handbook of Mathematical Fuzzy Logic - volume 1}, volume~37 of
  {\em Studies in Logic}, chapter \ Lukasiewicz logic and MV-algebras.
\newblock College Publications, London, 2011.

\bibitem{DiNola-Sessa}
A.~Di~Nola and S.~Sessa.
\newblock On {MV}-algebras of continuous functions.
\newblock {\em In: Non classical logics and their applications to fuzzy subsets
  (U. H\"{o}hle and E. P. Klement, Eds.),Kluwer Acad. Publ., Dordrecht}, pages
  23--32, 1995.

\bibitem{Gerla-2001}
B.~Gerla.
\newblock Rational {L}ukasiewicz logic and {DMV}-algebras.
\newblock {\em Neural Networks World}, 11:579--584, 2001.

\bibitem{Zaanen-1971}
W.A.J. Luxemburg and A.C. Zaanen.
\newblock {\em Riesz Spaces I}.
\newblock North-Holland, Amsterdam, 1971.

\bibitem{JM-2002}
J.~Martinez.
\newblock Hull classes of archimedean lattice-ordered groups with unit: a
  survey.
\newblock {\em In: Ordered algebraic structures,Kluwer Acad. Publ., Dordrecht},
  pages 89--121, 2002.

\bibitem{Montagna}
F.~Montagna.
\newblock Subreducts of {MV}-algebras with product and product residuation.
\newblock {\em Algebra Universalis}, 53:109--137, 2005.

\bibitem{Mundici-1986}
D.~Mundici.
\newblock Interpretation of {AF} {C}*-algebras in {L}ukasiewicz sentential
  calculus.
\newblock {\em J. Functional Analysis}, 65:15--63, 1986.

\end{thebibliography}
\bibliographystyle{plain}


\end{document}